\documentclass[10pt,a4paper]{article}

\usepackage{hyperref}

\usepackage{amsmath}
\usepackage{amsfonts}
\usepackage{amssymb}
\usepackage{amsthm}
\usepackage{threeparttable}
\usepackage{pdflscape}
\usepackage{amssymb}
\usepackage{graphicx}
\usepackage{setspace}
\usepackage[all]{xy}
\usepackage{dsfont}
\usepackage{multirow}

\newtheorem{thm}{Theorem}[section]
\newtheorem{lemma}[thm]{Lemma}
\theoremstyle{definition}
\newtheorem{remark}[thm]{Remark}

\newtheorem{definition}[thm]{Definition}

\newtheorem{alg}[thm]{Algorithm}
\usepackage{xcolor}
\newif\ifnotes\notestrue

\newcommand{\roll}{\mathbf{r}}
\newcommand{\parar}{\mathbf{s}}
\newcommand{\move}{\mathbf{m}}
\newcommand{\moveto}[1]{\mathbf{m}_{\text{to $#1$}}}
\newcommand{\resign}[1]{\mathbf{m}_{\text{$#1$}}}
\newcommand{\action}{\mathbf{a}}
\newcommand{\vr}{{V_{\roll}}}

\newcommand{\x}{\mathbf{x}}
\newcommand{\y}{\mathbf{y}}
\newcommand{\vect}{\mathcal{V}}

\newcommand{\vvv}{W}
\renewcommand{\H}{\mathbb{H}}
\newcommand{\K}{\mathbb{K}}

\newcommand{\statespace}{\mathcal{X}}
\newcommand{\X}{\mathcal{X}}
\newcommand{\A}{\mathcal{A}}

\newcommand{\indicator}{\,\mathbf{\chi}}
\newcommand{\R}{{\mathbb R}}

\newcommand{\Z}{{\mathbb Z}}

\newcommand{\I}{\mathcal{I}}
\newcommand{\actions}{\mathcal{A}}

\newcommand{\proba}{\operatorname{\bf P}}
\newcommand{\EE}{\operatorname{\bf E}}

\newcommand{\ind}[1]{\indicator_{#1}}

\newcommand{\ini}{\imath}
\newcommand{\fin}{\Delta}
\newcommand{\taupig}{\tau}
\begin{document}

\date{May 28, 2014}
\title{A finite exact algorithm to solve a dice game} 
\author{Fabi\'an Crocce\footnote{King Abdullah University of Science and Technology, Saudi Arabia \& Universidad de la Rep\'ublica, Uruguay} \ and Ernesto  Mordecki\footnote{Universidad de la Rep\'ublica, Uruguay}} 

\maketitle

\begin{abstract}
We provide an algorithm to find the value and an optimal strategy of the solitaire variant of the Ten Thousand dice game in the framework of Markov Control Processes.
Once an optimal critical threshold is found, the set of non-stopping states of the game becomes finite,
and the solution is found by a backwards algorithm that gives the values for 
each one of these states of the game.
The algorithm is finite and exact.
The idea to find the critical threshold comes from the \emph{continuous pasting condition} used in optimal stopping problems for {continuous-time} processes with jumps.
\end{abstract}


\section{Introduction}\label{section:intro}
The emergence of Probability Theory is closely related to the practice of dice games,
as masterly exposed by Hald \cite{Hald}.   
Inspired by this fact, and by the power of the mathematical method, we consider a popular dice game known as  ``Ten Thousand".  

In this game, played with five dice, several players, by turns, roll the dice several times.
Each possible outcome of a roll has an assigned score, {which may be zero}. 
If after rolling the dice, a strictly positive score is obtained, the player can roll again some dice to increase
his turn account, or he can stop rolling to bank his turn account into his general account, ending his turn; otherwise, if null score is obtained, the player looses the accumulated turn score, also ending his turn.

In consequence, each turn consists in a sequence of rolls, ending either when the player obtains no score
or when he decides to stop and bank his accumulated turn score.
Here arises the problem of taking an optimal decision.
The goal of the game is to be the first player in reaching a certain amount of points, usually 10000. 
This game, also known as Zilch and Farkle, among other names,  has several versions with minor differences (some of them played with six dice).

In this paper we consider the problem that faces an individual player who aims to maximize his turn score,
what can be considered a solitaire variant of the game. 
Modeling this optimization problem in the framework of Markov Control Processes (MCP),
we provide a finite exact algorithm that gives the value function and an optimal strategy for the considered game.
We expect that this algorithm can constitute a first step in finding optimal strategies for the original multi-player game. 

Despite the popularity of dice games and the interest of the probabilistic community in the topic, 
only a few references can be found concerning the family of dice games we consider.
Roters \cite{Roters} solves the solitaire version of the Pig game (a simpler variant with one dice, see subsection \ref{pig}). 
Shortly afterward, 
Haigh and Roters \cite{Haigh} consider a variant of the same solitaire game, 
consisting in minimizing the number of turns needed to reach a target.
Concerning competitive games, 
Neller and Presser \cite{Neller} solve the competitive version of the Pig game. 
More recently, Tijms \cite{Tijms} considered also the competitive Pig game and a simultaneous decision taking variant,
named Hog (see also \cite{Tijms1} where variants of the game are analyzed). 

As mentioned, the theoretical framework to study the solitaire game we consider is theory of Markov Control Processes.
Our problem happens to be a transient MCP under the expected total cost criteria,
with { countably infinite }
state space, finite number of actions at each state, and unbounded reward function.

Among the first papers considering the expected total cost criteria we find the paper by Blackwell (1967) \cite{blackwell2}. 
A general reference of both solitaire and competitive games with finite state space, 
can be found in the book by Filar and Vrieze \cite{Filar}.
Pliska \cite{pliska} considers MCP with infinite state space with bounded payoff function.
Our example does not fit in this framework, and similar simpler examples (see subsection \ref{pig}) show that
the Dynamic Programming Equation is not expected to have a unique solution.
This situation is considered by Hern\'andez-Lerma et al. \cite{lerma}, where an ad-hoc norm is proposed to restrict
the search space for the value function. Nevertheless, the conditions required in \cite{lerma} to prove the transience of
the process under all possible strategies, 
require certain uniformity condition that does not seem to be fulfilled in our case.

The rest of the paper is as follows. 
In section \ref{section:rules} we describe the rules of the game,
and in section \ref{section:solitaire} its mathematical model.
After writing the Dynamic Programming Equation (DPE) of the optimization problem, 
in section \ref{section:main} we present our main result, 
consisting in a finite algorithm that provides an exact solution to the DPE, 
resulting in the solution to the game. 
In section \ref{section:other} we solve some related games, and in section \ref{section:final} we present a brief discussion and some conclusions.
\section{Ten Thousand game}\label{section:rules}
{In the solitaire version of the Ten Thousand game that we consider, one player aims to maximize the score of one turn when playing the game described above.
This turn consists of a sequence of rolls.}
%

%
\subsection{Description of one roll.} 
The player rolls $n$ dice ($1\leq n\leq 5$). A score is assigned to each one of the dice combinations described in Table \ref{table:1}. 
\begin{table}
\begin{center}
\begin{tabular}{{l}{r}{r}|{l}{r}{r}}
\hline
Combination & Points & Chips & Combination & Points & Chips \\
\hline
Each 5& $50$ & $1$ & Three 4s&$400$ & $8$ \\
Each 1& $100$  & $2$ & Three 5s&$500$ & $10$ \\ 
Three 2s&$200$ & $4$ & Three 6s&$600$& $12$ \\  
Three 3s&$300$ & $6$ & Three 1s&$1000$ & $20$ \\ 
\hline
\end{tabular}
\end{center}
\caption{Combination of dice and corresponding scores in the Ten Thousand dice game}
\label{table:1}
\end{table}
We introduce a \emph{dice configuration}  $i=[f,o,t]$  
to summarize the outcome of a roll that has $f$ fives, $o$ ones and three $t$-s for $t=2,3,4,6$. By $t=0$ we denote the absence of three of a kind for these numbers.

Mathematically equivalent, to simplify the treatment, we divide all point amounts by 50. 
To link with the real game, the reader can consider ``big points" or ``chips" costing 50 points each.
Observe that three fives or three ones are counted in $f$ and $o$ respectively. 
Denote by $\mathbf{0}=[0,0,0]$ the non-scoring configuration.
The set  $\I$  of configurations with non vanishing score is listed in Table \ref{table:freq}. 			

To each configuration $i=[f,o,t]$ we associate a score $s(i)$,
a number of scoring dice $d(i)$,
and a number of scoring combinations $c(i)$,
through the formulas:
\begin{align*}
s(i)&=f+2o+14\indicator_{\{o\geq 3\}}+7\indicator_{\{f\geq 3\}}+2t,\\
d(i)&=f+o+3\indicator_{\{t\neq 0\}},\\
c(i)&=f+o+\indicator_{\{t\neq 0\}}.
\end{align*}
Observe that three fives or three ones count as three scoring combinations.


\begin{definition}\label{def:smaller} We say that a configuration $j$ is \emph{smaller} than a configuration $i$ when the set of scoring combinations of $j$ is a non-void strict subset of the set of scoring combinations of $i$. In this case we write $j\prec i$ and have $s(j)<s(i)$, $d(j)<d(i)$, $0<c(j)<c(i)$.
\end{definition}

\subsection{Description of the game.}
We denote by $i_1,\dots,i_k,\dots$ the successive outcomes of a sequence of rolls, by $\tau_k$ the points accumulated up to the $k$-th dice roll, and by $n_k$ the amount of dice available to roll after the $k$-th roll. 
The values of $\tau_k$ and $n_k$ depend on: $\tau_{k-1}$, the random outcome $i_k$, and the action chosen by the player, as follows.
In the first roll the player begins with zero points and five dice to roll, i.e. $\tau_0=0$ and $n_0=5$.
Depending on the configuration $i_k$ obtained in the $k$-th roll,
the player has the following actions to take:
\begin{itemize}
\item[(i)] If no score is obtained ($s(i_k)=0$) the player looses the accumulated score ($\tau_k=0$) ending the game.
\item[(ii)] If all rolled dice give score ($d(i_k)=n_{k-1}$) then the score increases to 
$\tau_k=\tau_{k-1}+s(i_k)$ and
the player can either stop, earning $\tau_k$ points and finishing the game, or he can roll the five dice anew ($n_k=5$). 
\item[(iii)] If not all rolled dice give score and only one scoring combination is obtained ($d(i_k)<n_{k-1}$ and $c(i_k)=1$) then the score increases to $\tau_k=\tau_{k-1}+s(i_k)$ and the player
can either stop, earning $\tau_k$ points, or he can roll $n_k=n_{k-1}-d(i_k)$ dice. 
\item[(iv)] If not all rolled dice give score and more than one scoring combination is obtained ($d(i_k)<n_{k-1}$ and $c(i_k)>1$) then, in addition to the actions described in (iii), the player can choose a configuration $j_k\prec i_k$ (see Definition \ref{def:smaller}),  
increasing the score only to $\tau_k=\tau_{k-1}+s(j_k)$, but 
obtaining $n_k=n_{k-1}-d(j_k)$ dice to roll (the player resigns score to roll more dice).
\end{itemize}
Table \ref{table:outcomes} summarizes the above described actions.
After each roll, the player goes on rolling the non-scoring dice ({and eventually} part of the scoring-dice in case (iv)), 
taking decisions  according to (ii)-(iv),
until he decides to stop, or until he gets no scoring dice (i). 
The goal of the game is to maximize the accumulated score at the end of the game.
Observe that, unless situation (ii) is achieved,
the number of rolling dice of the subsequent rolls is strictly decreasing, as at each step some scoring dice should be set aside. 
It should be noticed that any time {the player} decides to roll, he risks losing the accumulated turn account and finishing the game with 0 points.

\begin{table}
\begin{center}
\begin{tabular}{c|c|c|c}

$s(i)$ & $d(i)$ & $c(i)$ & Actions \\
\hline
\hline
0 & 0 & 0	&	Stop	\\
\hline
\multirow{3}{*}{$>0$} & $n$ 	&	&Stop / Roll 5 dice	\\ \cline{2-4}
&	\multirow{2}{*}{$<n$}& $1$ & Stop / Roll $n-d(i)$ dice	\\ \cline{3-4}
&	& $>1$ & Stop / Roll $n-d(i)$ dice / Roll $n-d(j)$ dice $(j\prec i)$	\\
\hline
\end{tabular}
\end{center}
\caption{Possible actions after rolling $n$ dice and obtaining configuration $i$}
\label{table:outcomes}
\end{table}

\begin{remark}\label{remark:1}
Observe that, according to (ii), the player can in principle roll the dice an arbitrary number of times, accumulating an arbitrary large score.
\end{remark}

%
%
%
%

\section{Mathematical model of the solitaire game}\label{section:solitaire}

{We model the Solitaire Ten Thousand game as a Markov Control Process
in which one player aims to maximize the non-discounted sum of the rewards. 
To this end we define the state space $\statespace$, 
the sets $\actions(\x)$ of actions the player can take at each state $\x\in\statespace$, and for each pair $(\x,\action)$ with $\action \in \actions(\x)$ we define a reward $r(\x, \action)$ and a probability distribution $q(\cdot|\x,\action)$ over the state space $\statespace$ for the following state.
\subsection{\bf States}
Each state $\x$ of the solitaire game includes three elements: 
an amount $\tau$ of accumulated chips;
the last obtained configuration $i=[f,o,t]$; and
a number $n=1,\dots,5$ of available dice to roll. 
It is necessary to add two singular states:
the initial state $\ini=(0,\emptyset,5)$ and
a final absorbing state $\fin=(0,\mathbf{0},0)$.

In this way we define the state space of the MCP process as 
$$
\statespace=\{\ini,\fin\} \cup \left\{\x=(\tau,i,{n})\colon\tau \in \{1,2,\dots\}, i\in \I,n \in \{1,2,3,4,5\}\right\}.
$$ 
We observe that the state space of the game is a countably infinite set, in accordance with Remark \ref{remark:1}.
\subsection{\bf Actions, transition probabilities and rewards}\label{rules}
When from a state $\x=(\tau,i,{n})$ all the $n$ available dice are rolled (action that we call \emph{roll} and denote by $\roll$) the probability distribution for the next state $\y$ is determined by 
\begin{equation}\label{eq:rolltransitions}
q\left(\y|\x, \roll \right)=\begin{cases}
{f(n,\mathbf{0})/ 6^{n}}, &\text{if $\y=\fin$}\\
{f({n},j)/ 6^{n}},	&\text{ if $\y=\left(\tau+s(j),j,n-d(j)+5 \ind{\{n=d(j)\}}\right)$},\\
0,					&\text{ otherwise,}
\end{cases}
\end{equation}
}
where $f({n},j)$ are the frequencies listed in Table \ref{table:freq}. 
Observe that the indicator $\ind{\{n=d(j)\}}$ models the right to throw five dice anew 
when all dice give score (case (ii) above). The previous formula also applies in the case $\x=\ini$, in which $x=0$ and $n=5$.

There are several actions the player can choose:  
to roll all the available dice $\action=\roll$, 
to stop the game $\action=\parar$, and to ``make a move" $\action=\move_{\bullet}$ (which could be a set of actions). We follow by describing the respective actions associated with each state, 
with their corresponding rewards and probability transitions. 
\begin{enumerate}
\item
In the initial state $\ini=(0,\emptyset,5)$ the player has the only action  $\roll$, with null reward and probability transitions $q(\cdot|\ini,\roll)$ given in \eqref{eq:rolltransitions}.
\item
In a state $\x=(\tau,i,n)$ with one scoring combination, (i.e. $c(i)=1$), or with five dice to roll ($n=5$) 
the player has two actions: 
  $\parar$, with reward $\tau$ and probability one to go to $\fin$; 
and  $\roll$, which has null reward and probability transitions $q(\cdot|\x,\roll)$ given in \eqref{eq:rolltransitions}. 
\item
In a state $\x=(\tau,i,n)$ with more than one scoring combination, (i.e. $c(i)>1$) and $n<5$, 
the player has, besides the same two actions described in item 2, the possibility of making a \emph{move} to  a state with a smaller configuration $j\prec i$ (see Definition \ref{def:smaller}), this action is denoted by $\moveto{j}$. The move $\moveto{j}$ has null reward, and probability one to go to the state $(\tau-s(i)+s(j),j,n+d(i)-d(j)).$
\end{enumerate}
It should be noticed that the only non-null reward corresponds to the action $\parar$, 
and after taking this action the game ends. Thus we have a MCP under the total cost
criterion, where
$$
\sum_{j=1}^{\infty}r(\x_j,\action_j)=\sum_{j=1}^{\infty}\tau_j\indicator_{\{\action_j=  \parar\}}.
$$
Therefore, the objective function has only one addend 
(i.e. the accumulated score at the stopping moment), or vanishes after a non-scoring roll.

\begin{remark} \label{rem:move}
About the family of actions $\move$, observe that starting from a state $(\tau,i,n)$ if configurations $i,j,k$ satisfy $k\prec j\prec i$, moving from $i$ to $k$, going with full probability to $(\tau-s(i)+s(k),k,n+d(i)-d(k))$, is equivalent to move from $i$ to $j$, going to $(\tau-s(i)+s(j),j,n+d(i)-d(j))$, and in the following step moving from $j$ to $k$ going to $(\tau-s(i)+s(j)-s(j)+s(k),k,n+d(i)-d(j)+d(j)-d(k))$.
The previous observation would allow us to consider only movements that resigns one scoring combination.
\end{remark}

The possible moves $\move_{\bullet}$ the player can make from a given state can be described by enumerating the scoring dice the player resign, i.e. the difference between the original configuration $i$ and the new configuration $j$. With this in mind the number of possible actions is reduced to 15, which are described in Table \ref{table:actions}. 
\begin{table}
\label{table:actions}
\begin{center}
\begin{tabular}{r|l}
Notation & Description \\
\hline
$\parar$ & stop\\
$\roll$ & roll all the non-scoring dice\\
$\resign{5}$ & resign one 5 \\
$\resign{1}$ & resign one 1 \\
$\resign{55}$ & resign two 5 \\
$\resign{51}$ & resign one 5 and one 1 \\
$\resign{11}$ & resign two 1 \\
$\resign{551}$ & resign two 5 and one 1 \\
$\resign{511}$ & resign one 5 and two 1 \\
$\resign{t_k}: k \in 1..6$ & resign a tern of $k$\\
\end{tabular}
\end{center}
\caption{Possible actions}
\end{table}

The list of configurations with their corresponding actions is summarized in Table \ref{table:freq}. 

\subsection{Dynamic Programming Equation}\label{section:bellman}
Consider $\K:=\left\{ (\x,\action ) \colon \x \in \statespace,\ \action \in \actions(\x) \right\}$;
 $\H_0:=\statespace$; and $\H_n:=\K^n \times \statespace$. 
 An element $h_n=(\x_0,\action_0, \ldots,\x_n)\in \H_n$ is called an $n$-history. 
 A strategy (or control policy) is a sequence $\pi=(\pi_1,\pi_2,\ldots)$,
 where $\pi_n$ is a stochastic kernel defined on $\H_n$ such that $\pi_n(\cdot|h_n)$ is supported on $\actions(\x_n)$. 
 If for every $n$ and for every $h_n \in \H_n$, the measure $\pi_n(\cdot|h_n)$ is supported in just one action $\action \in \actions(\x_n)$, the strategy is said to be \emph{pure}, and each $\pi_n$ can be seen as a function that, given the history, chooses the following action. 
 A strategy $\pi$ is said to be \emph{Markovian }if it satisfies $\pi_n(\cdot|h_n)=\pi_n(\cdot|h'_n)$ provided that the last state in both histories coincide, in this case the notation $\pi_n(\cdot|\x_n)$ is used instead of $\pi_n(\cdot|h_n)$. 
 A Markovian strategy is said to be \emph{stationary }if $\pi_n(\cdot|\x_n)=\pi_m(\cdot|\x_m)$ provided that $\x_n=\x_m$; in this case one can avoid the $n$ in the notation using just $\pi(\cdot | \x_n)$. 

\begin{remark} \label{rem:purestrategy}
Denoting $\actions=\cup_{x\in \statespace}\actions(x)$, a pure stationary strategy $\pi$ can be understood as function $f\colon \statespace \to \actions$, that satisfies the constraint $f(\x)\in \actions(\x)$ for all $\x\in \statespace$, where $f(\x)$ is the action that has full probability to be picked when the system is in the state $\x$, i.e. $\pi(f(\x)|\x)=1$.
\end{remark}

Given a strategy $\pi$ and an initial state $\x \in \statespace$, a probability measure $\proba^{\pi}_{\x}$ on  $\K^\infty$ with the product $\sigma$-algebra is determined in a canonical way. Denoting by $\EE_\x^\pi$ the expectation under $\proba_\x^\pi$, the player expects to receive, when starting from $\x$ and following the strategy $\pi$,
$$
V(\x,\pi)=\EE_\x^\pi\left(
\sum_{j=0}^{\infty}r(\x_j,\action_j)
\right)
$$
For a more detailed exposition and subclasses of strategies we refer to \cite{lerma}.
Denoting by ${\Pi}$ the set of all possible strategies, 
our objective is to find the value function $V^*\colon \statespace \to \R$ and an optimal strategy ${\pi}^*\in{\Pi}$ such that
\begin{equation}\label{eq:value}
V^*(\x)=\sup_{\pi\in\Pi} V(\x,\pi)=V(\x,\pi^*),\quad \text{for all $\x\in\X$}.
\end{equation}

The DPE of the solitaire game is,
\begin{equation}\label{eq:bellman}
V(\x)=\max_{\action\in\A(\x)}\left\{
r(\x,\action)+
\sum_{\y \in \statespace} V(\y)\ q(\y|\x,\action)
\right\}
\end{equation}
which, introducing the notation
\begin{equation}\label{eq:vroll}
\vr(\tau,{n})=\sum_{k\in\I}\frac{f(n,k)}{6^n}V\left(\tau+s(k),k,{n}-d(k)+5\ind{\{n=d(k)\}}\right),
\end{equation}
can be written as
\begin{equation}
\tag{\ref{eq:bellman}*}
\label{eq:bellmanexplicit}
V(\tau,i,n)= \tau\vee \vr(\tau,{n})
\vee
\ind{\{n\neq 5\}}\left[\max_{j\prec i}
V(\tau -s(i)+s(j),j,{n}+d(i)-d(j))\right].
\end{equation}

\section{Main result}\label{section:main}

We state our main result, whose proof is at the end of this section.
\begin{thm} \label{teo}
{\rm (a)} 
There exists a pure and stationary optimal strategy $\pi^*$ for the considered game. \\
{\rm (b)}
For states $\x=(\tau,i,n)$ with $\tau \geq 56$, the value function is $V^*(\x)=\tau$, and the optimal action is $\parar$.\\
{\rm (c)}
For the rest of the states, the value (and optimal action) is obtained exactly by Algorithm \ref{alg} in an efficient way, in which most of the states get their value in one shot, and for the others just a few updates are needed. The value of the game (from its initial state) is $V^*(\ini)=5.872019$. 

Table \ref{table:value} summarizes the value function and optimal action for every state $\x=(\tau,i,n)$ with $\tau < 56$. We remark that actions $\resign{11},\resign{551},\resign{511},$ and $\resign{t_k}: k \in 1..6$ are never optimal.

\end{thm}

Notations and remarks on the algorithm:
\begin{itemize}
\item $\I(n)=\{i\in \I \colon d(i)\leq n\}$ i.e. are the scoring configuration that can be obtained after rolling $n$ dice. 
\item $\I_0(n)=\{i\in \I(n) \colon c(i)=1\}$ i.e. are the configuration with just one scoring combination that can be obtained after rolling $n$ dice. The only available actions in a state with a configuration in $\I_0(n)$ are $\parar, \roll$.
\item $\I(i,n)=\{j\in \I\colon i\prec j,\ d(j)<d(i)+n\}$. Are the configurations from which the player can decide to move to a state $(\tau,i,n)$.
\item $\vr(\tau,n)$ should be computed following \eqref{eq:vroll} considering $V(\tau,i,n)=\tau$ for $\tau\geq 56$.
\item After the execution of the algorithm $V$ will have the optimal value and $A$ will have the optimal action.
\item For notation convenience we use $\vect(\tau,i,n)$ to denote $(V(\tau,i,n),A(\tau,i,n))$
\end{itemize}

\begin{alg}\label{alg}
\texttt{
\begin{enumerate}
\item for $\tau=55$ {downto} $0$ 
\item \qquad for $i\in\I$
\item \qquad\qquad $\vect(\tau,i,5)\gets (\vr(\tau,5),\roll)$
\item \qquad for ${n}=4$ {downto} $1$ 
\item \qquad\qquad $aux \gets \vr(\tau,n)$
\item \qquad\qquad for $i\in \I(5-n)$ 
\item \qquad\qquad\qquad if $aux\leq\tau$ then
\item \qquad\qquad\qquad\qquad $\vect(\tau,i,n)\gets (\tau,\parar)$
\item \qquad\qquad\qquad else
\item \qquad\qquad\qquad\qquad $\vect(\tau,i,n)\gets (aux,\roll)$
\item \qquad\qquad\qquad\qquad for $j\in \I(i,n)$
\item \qquad\qquad\qquad\qquad\qquad if $V(\tau-s(i)+s(j),j,n-d(j)+d(i))< aux$ then
\item \qquad\qquad\qquad\qquad\qquad\qquad $\vect(\tau-s(i)+s(j),j,n-d(j)+d(i))\gets (aux,\moveto{i})$
\end{enumerate}
}
\end{alg}

In the appendix a detailed analysis of the DPE is included together with a more explicit and efficient version of the algorithm.

To prove that the previous algorithm actually arrives to the solution of the game we need some previous considerations. 

Denote by $F(\X,\R)$ the set of real functions defined on $\X$, and consider the operator 
$U\colon F(\X,\R)\to F(\X,\R)$ acting by
\begin{equation}
\label{eq:defU}
UV(\x)=\max_{\action \in\A(\x)}U_\action V(\x),
\end{equation}
with 
$$U_\action V(\x)=\left\{r(\x,\action)+\sum_{\y \in \statespace} V(\y)q(\y|\x,\action)\right\}.$$
Let $\vvv_0\in F(\X,\R)$ denote the null function ($\vvv_0(\x)=0,\ \forall \x \in \X$).
From Theorem 3 in \cite{blackwell2} we know that the value function $V^*$ is the smallest solution of equation \eqref{eq:bellman} and it satisfies $V^*(\x)=\lim_{k\to\infty}\vvv_k(\x)$, where  
$\vvv_k=U^{(k)}\vvv_0=(U\circ U\circ\cdots\circ U)\vvv_0$ is the $k$-th iteration of the operator $U$ applied on $\vvv_0$.

To construct a finite algorithm, given the fact that the number of states is infinite, we need to reduce the number of relevant states to a finite amount. 
To this end, a first step in the algorithm design is the determination of the states in which is optimal to stop immediately, i.e. where $V^*(\tau,i,{n})=\tau$.  
The search is based on the fact that for large values of $\tau$ we expect not to be optimal to risk the accumulated score, so the optimal action would be to stop and $V(\tau,i,{n})=\tau$. 
We mimic the search of solution of optimal stopping problems for continuous time processes with positive jumps, 
that verify integral equations (see for instance \cite{osdij}). 
Assuming that there exists a critical threshold $\tau^*$ such that
$$
\tau^*=\min\{\tau\colon V(\tau,i,{n})=\tau \quad\forall (i,{n})\},
$$
and based on the nature of the game, it is natural to suppose that the minimum is attained in a state with $n=5$.
We then compute
\begin{equation}\label{eq:56}
\tau^*(5)=\min \left\{\tau\in \Z^+ \colon \tau\geq \sum_{k\in\I}\frac{f(5,k)}{6^5}\left(\tau+s(k)\right)\right\},
\end{equation}
or, what is the same, 
$$\tau^*(5)=\min \left\{\tau\in \Z^+ \colon \sum_{k\in\I}\frac{f(5,k)}{6^5} s(k)\leq \tau \frac{f(5,\emptyset)}{6^5} \right\},$$
obtaining $\tau^*(5)=56$. 

In the following lemma we give a detailed proof of this intuitive result. 

\begin{lemma}
\label{lemgeq56}
If $\x=(\tau,i,n)$ with $\tau\geq 56$ then $V^*(\x)=\tau$.
\end{lemma}
\begin{proof}
The proof is based in the fact that $V^*(\x)=\lim_{k\to\infty}\vvv_k(\x)$. First, observe that for any $\x=(\tau,i,n)$ we have $\vvv_1(\x)=x$. To see this, remember that $\vvv_1=U \vvv_0$, which, according to \eqref{eq:defU} is
\begin{eqnarray*}
\vvv_1(\x)&=&\max_{\action \in\A(\x)}U_a\vvv_0(\x)
\end{eqnarray*}
The maximum is attained with the action $\action=\parar$, which is the only one such that $U_a\vvv_0(\x)>0$. In fact, $U_\parar \vvv_0 (\x)=r(\x,\parar)=\tau$. We follow by proving by induction that for any $k\in \Z^+$, $\vvv_k(\x)<56$, if $\x=(\tau,i,n)$ with $\tau<56$, and $\vvv_k(\x)=\tau$, if $\x=(\tau,i,n)$ with $\tau\geq 56$: For $k=1$ is already proved. Assume the result holds for $k<m$. We have 
\[ 
\vvv_m(\x)=\max_{\action \in\A(\x)}\left\{U_a\vvv_{m-1}(\x)\right\}
\]
If $\x=(\tau,i,n)$ with $\tau \geq 56$, let us see that the maximum is attained with the action $\parar$ (giving $\vvv_m(\x)=\tau$): 
An action $\moveto{j}$, if possible, would imply a transition with full probability to the state $\y=(\tau-s(i)+s(j),j,n+d(i)-d(j))$, where $\tau-s(i)+s(j)<\tau$. Then we have $U_{\moveto{j}}\vvv_{m-1}(\x)=\vvv_{m-1}(\y)$, which is less than $\tau$ by induction hypothesis. On the other hand, for the action $\roll$ we have  
\begin{equation}
U_{\roll} \vvv_{m-1}(\x) = \tau + \sum_{k\in\I(n)}\frac{f(n,k)}{6^n} s(k) - \tau \frac{f(n,\emptyset)}{6^n}<\tau,
\end{equation}
since for $\tau\geq 56$
\begin{equation}
\label{eq:56n}
\sum_{k\in\I(n)}\frac{f(n,k)}{6^n} s(k) - \tau \frac{f(n,\emptyset)}{6^n} < 0
\end{equation}
not only for $n=5$ but also for $n=1,2,3,4$ (this fact can be verified directly).

If $\x=(\tau,i,n)$ with $\tau < 56$ then $U_a\vvv_{m-1}(\x)<56$ for all possible actions. For action $\parar$ and $\move$ is a direct consequence of the induction hypothesis, while for the action $\roll$ we have
\begin{eqnarray*}
U_{\roll} \vvv_{m-1}(\x) & = &\sum_{k\in\I(n)}\frac{f(n,k)}{6^n}\vvv_{m-1}\left(\tau+s(k),k,{n}-d(k)+5\ind{\{n=d(k)\}}\right),\\
& \leq &\sum_{k\in\I(n)}\frac{f(n,k)}{6^n}(56+s(k)),\\
& =& U_{\roll} \vvv_{m-1}((56,i,n))\\
& \leq & 56.
\end{eqnarray*}

Now that we know $\vvv_m(\x)=\tau$, for $\x=(\tau,i,n)$ with $\tau \geq 56$, for all $m\geq 1$, we can take limit to conclude that $V(\x)=\tau$.
\end{proof}

\begin{proof}[Proof of Theorem \ref{teo}]
To prove claim {\rm (a)}, one just need to observe that, as the set $A(\x)$ is finite for each $\x$, a function $f\colon \statespace \to \actions$ can be defined such that
\begin{equation*}
\max_{\action\in\A(\x)}\left\{r(\x,\action)+\sum_{\y \in \statespace} V^*(\y)q(\y|\x,\action)\right\}=r(\x,f(\x))+\sum_{\y \in \statespace} V^*(\y)q(\y|\x,f(\x)),
\end{equation*}
for all $\x \in \statespace$. This function $f$ defines a pure stationary strategy (see remark \ref{rem:purestrategy}) which clearly attains supremum in \eqref{eq:value}, this meaning that the strategy is optimal.

Claim {\rm (b)} is already proved in Lemma \ref{lemgeq56}.

Let us prove claim {\rm (c)}. By {\rm (b)} we know $V^*(\tau,i,n)$ with $\tau\geq 56$. 
To find $V^*$ we need to solve equations \eqref{eq:bellmanexplicit} for states $\x=(\tau,i,n)$ with $\tau\leq 55$ under the boundary condition $V^*(\tau,i,n)=\tau$ for $\tau\geq 56$. 
This gives a unique solution that is obtained by our algorithm by backward iteration, starting with states with $\tau=55$. 
Consider the following affirmation, valid for $k=1\dots 56$: After the $k$-th iteration of the ``\texttt{for $\tau=55$ {downto} $0$}", in which $\tau=56-k$:
\begin{itemize}
\item[(A1)] $V(\x)$, for states $\x=(\tau,i,n)$ have the value $\tau \vee V_r(\tau,n)$. In the case of states with only two actions ($\actions(\x)=\{\parar,\roll\}$) this value coincide with $V^*(\x)$ and it is not modified again during the algorithm.
\item[(A2)] $V(\x)$, for states $\x=(\tau-1+s(i),i,n)$ and $i\in \I$ will have its definitive value, which coincides with $V^*(\x)$.
\end{itemize} 

For $k=1$ ($\tau=55$): (A1) is a consequence of the fact that $V_{\roll}(55,n)$ (computed in the lines 3. and 5. according to \eqref{eq:vroll}) depend only on values of $V$ in states with number of accumulated chips greater than $55$, which are known by the boundary condition. In lines 2-10 of the algorithm $V(55,i,n)$ gets the maximum between this known value and $\tau=V_{\parar}(\x)$. If the only available actions are $\parar$ and $\roll$ this value is $V^*(55,i,n)$. (A2) holds since we are considering states $\x=(54+s(i),i,n)$; if $s(i)\geq 2$ this values fall in the boundary condition while if $s(i)=1$ there is only one scoring combination, so the only available actions are $\parar$, and $\roll$ and we are in case (A1).

Consider $l\leq 56$. Assuming the affirmation is valid for $k=1\dots l-1$ let us prove it for $k=l$, $(\tau=56-l)$: to verify the validity of (A1) one just need to observe that in the computation of $V_{\roll}(56-l,n)$ only known values of $V$ (which already have their definite value) are needed. 
Actually, $V_{\roll}(56-l,n)$ depends on $V(\y)$ for states of the form $\y=(56-l+s(j),j,n-d(j))$, $j\in \I$, and by hypothesis these are already known values. 
To prove (A2), observe that in the iteration $l$, lines 11-13 of the algorithm, states from which a move to the currently considered state is possible are updated if this move gives a better value than the current one. With this in mind, we observe that a state with actions move available, gets its definite value after considering all the possible destination of the move. The last destination considered is the one with smallest $\tau$. Now observe that from a state $\x=(56-l-1+s(i),i,n)$ the possible moves are to states $\y=(56-l-1+s(j),j,n-d(i)+d(j))$, with $j\prec i$; as $s(j)\geq 1$ we have $56-l-1+s(j)\geq 56-l$ concluding that after the $l$-th iteration $V(\x)$ will have its definite value.

From the proved affirmation we obtain that after all the 56 iterations $V(\x)=V^*(\x)$ for all the states with $\tau\leq 55$.
\end{proof}

\section{Other related games}\label{section:other}
\subsection{Ten Thousand with restricted actions}
A natural game related with the one studied in this paper is the \emph{stop or roll} Solitaire Ten Thousand.  
It is essentially the same game but without the possibility of taking action $\move$ (the player can only {take the actions   $\parar$ and  $\roll$}). It has value $V_{{S/R}}=5.576326$ and the algorithm to solve it is the same algorithm \ref{alg} without the lines \verb|9|, \verb|10| and \verb|11|. 

Other games obtained restricting the total number of possible actions of the original game can also be solved. The values (from the initial state) of some of this games are presented in Table \ref{table:nr}. See the appendix for more details on the algorithm to find them.

\begin{table}
\begin{center}
\begin{threeparttable}
\begin{tabular}{c|c|c}
Nr. of actions & Possible actions\tnote{1} & Value of the game \\
\hline
2 & { $\parar$, $\roll$} & 5.5763262782 \\
3& {  $\parar$, $\roll$}, $\resign{5}$ & 5.8012180037\\
4& {  $\parar$, $\roll$}, $\resign{5}$, $\resign{1}$ & 5.8153340639\\
5&  {  $\parar$, $\roll$}, $\resign{5}$, $\resign{1}$, $\resign{55}$ & 5.8707484326\\
6&  {  $\parar$, $\roll$}, $\resign{5}$, $\resign{1}$, $\resign{55}$, $\resign{51}$ & 5.8720189185\\
15 & all    & 5.8720189185\\
\end{tabular}
\end{threeparttable}
\end{center}
\caption{Values of restricted actions games}
\label{table:nr}
\end{table}

\subsection{The solitaire Pig game}\label{pig}
The solitaire Pig game is similar, played with only one dice. The only options the player has is to roll again or to stop (for details see \cite{Roters}). The state can be modeled with just $\taupig$, the accumulated score.
The problem is in this case an optimal stopping problem.
The corresponding DPE is
\begin{equation}\label{eq:pig}
V(\taupig)=\max\left\{\taupig,\frac16\left(V(\taupig+2)+V(\taupig+3)+V(\taupig+4)+V(\taupig+5)+V(\taupig+6)\right)\right\}
\end{equation}
The critical threshold is $\taupig^*=20$. Computing backwards, we obtain $V^*(0)=8.14$. 
In this example, we find instructive to observe that the difference equation 
$$
V(\taupig)=\frac16\left(V(\taupig+2)+V(\taupig+3)+V(\taupig+4)+V(\taupig+5)+V(\taupig+6)\right)
$$
has infinite number of solutions, under the condition $V(\taupig)\geq \taupig$, for instance $V(\taupig)=V(0){1.0646}^\taupig$ for all large enough $V(0)$. The value function is the minimal solution, which satisfies $V(\taupig)=\taupig$ for all $\taupig\geq \taupig^*$.

\section{Conclusions}\label{section:final}

The theory of Markov Control Processes is a powerful tool to analyze dice games. 
Nevertheless, although in general it is possible to write the DPE of the game,
in many concrete situations the number of states and actions makes very difficult to solve effectively the game. 
In these situations, 
when possible, 
it is necessary to take into account
the particular characteristics of the  game in order to solve the problem. 
Usually, these equations are solved by iteration. 
In the present paper we find a very simple and exact algorithm to solve the Solitaire Ten Thousand.
It is interesting to note that the value function is characterized as the smallest solution to the DPE,
in a framework where the uniqueness of the solution to this equation is not assured.
The same idea is also used to solve some related simpler games. 
Whether a similar type of algorithm can be used in the competitive Ten Thousand game 
remains an open question. 

\section*{Appendix: Detailed analysis of the DPE and a more efficient algorithm}

In this section we make a detailed analysis of the DPE (family of equations \eqref{eq:bellman}) to obtain an algorithm more adapted to this particular game. The idea is to gather configurations for which the (relevant) set of actions coincide and compute the value function for all of them together. 

First observe that
$
\sup_{{\pi}\in{\Pi}} V((\tau,i,{n}),{\pi})\leq\sup_{{\pi}\in{\Pi}} V((\tau+1,i,{n}),{\pi}),
$
since, given the fact that the set of available actions depend only on $i$ and $n$, 
any strategy departing from $(\tau,i,n)$ can also be applied departing from $(\tau+1,i,n)$ with the same transition scheme. From this observation we get
\begin{equation}
V^*(\tau,i,{n})\leq V^*(\tau+1,i,{n}).\label{eq:tau}
\end{equation}

To start consider a state $\x=(\tau, i, n)$ such that $c(i)=1$. Equation \eqref{eq:bellmanexplicit} for these states is
$$V(\tau,i,n) = \tau \vee \vr(\tau,{n}). $$
Looking at the right hand side we conclude that the value does not depend on $i$, and use the notation $V_0(\tau,n)$ to represent it. In the case $n=5$ we have the same equation for any $i$, so the same consideration applies to this case and we use the notation $V_0(\tau,5)$ to represent the value function of any state of the form $(\tau,i,5)$. In summary we have 

\begin{equation}\label{eq:sr}
V_0(\tau,n)=\tau\vee \vr(\tau,n),
\quad
\text{for $n=1,2,3,4,5$}
.\tag{$E_0$}
\end{equation}

Observe that all the states with $n=4$ fall into the previous family of equations, since the configuration $i$ has to have just one scoring die.

Let us analyze configurations with two scoring dice. With the same idea of gathering configurations, we use the notation $V_5(\tau,n)$ to refer to $V(\tau,i,n)$ for states such that $\parar,\roll,\resign{5}$ are the available actions, i.e. states with $n=1,2,3$ and $i=[2,0,0]$ (we also include here the case $i=[1,1,0]$, because despite it is also possible the action $\resign{1}$ it is clearly worse than $\resign{5}$ by \eqref{eq:tau}). The DPE in this case is 
\begin{align*}\label{eq:sr5}
V_5(\tau,n)&=V_0(\tau,n)\vee \vr(\tau-1,n+1)\quad\text{for $n=1,2,3$}
.\tag{$E_5$}
\end{align*}
For $i=[0,2,0]$, we have
\begin{align*}
\label{eq:sr1}
V(\tau,[0,2,0],n)&=V_0(\tau,n)\vee \vr(\tau-2,n+1),
\quad\text{for $n=1,2,3$}
.\tag{$E_1$}
\end{align*}

For states whose configuration has three scoring dice we have
\begin{align}
V(\tau,[2,1,0],n)&=V_0(\tau,n)\vee V_5(\tau-1,n+1),\quad\text{for $n=1,2$}.\tag{$E_{55}$}\label{eq:sr55} \\
V(\tau,[1,2,0],n)&=V_0(\tau,n)\vee V(\tau-1,[0,2,0],n+1),\quad\text{for $n=1,2$}.\tag{$E_{51}$}\label{eq:sr51} \\
V(\tau,[3,0,0],n)&=V_0(\tau,n)\vee V_5(\tau-8,n+1),\quad\text{for $n=1,2$}\tag{$E_{t_5}$}\\
V(\tau,[0,3,0],n)&=V_0(\tau,n)\vee V(\tau-16,[0,2,0],n+1),\quad\text{for $n=1,2$}.\tag{$E_{t_1}$}
\end{align}
As we do for \eqref{5} and \eqref{1}, all the right hand sides can be expressed as a maximum between $\tau$ and $V_r$.

It only remains to consider states with $n=1$ and configurations $i$ with four scoring dice. They are
\begin{align}
V(\tau,[2,2,0],1)&=V_0(\tau,1)\vee V(\tau-1,[1,2,0],2),\tag{$E_{551}$}\label{e551}\\
V(\tau,[1,3,0],1)&=V_0(\tau,1)\vee V(\tau-1,[0,3,0],2),\tag{$E_{5t_1}$}\label{e5t1}\\
V(\tau,[1,0,t],1)&=V_0(\tau,1)\vee V_0(\tau-1,2)\vee V_0(\tau-2t,4),\tag{$E_{5t}$}\label{e5t}\\
V(\tau,[4,0,0],1)&=V_0(\tau,1)\vee V(\tau-1,[3,0,0],2),\tag{$E_{5t_5}$}
\end{align}
\begin{align}
V(\tau,[3,1,0],1)&=V_0(\tau,1)\vee V(\tau-2,[3,0,0],2)\vee V(\tau-8,[2,1,0],2),\tag{$E_{1t_5}$}\label{e1t51}\\
V(\tau,[0,4,0],1)&=V_0(\tau,1)\vee V(\tau-2,[0,3,0],2)\tag{$E_{1t_1}$}\label{e1t1}\\
V(\tau,[0,1,t],1)&=V_0(\tau,1)\vee V_0(\tau-2,2)\vee V_0(\tau-2t,4).\tag{$E_{1t}$}\label{e1t}
\end{align}
Each group of equations \eqref{e5t} and \eqref{e1t} comprises four equations,
as $t$ ranges in the set $\{2,3,4,6\}$. In some of the previous equations we have used \eqref{eq:tau} and also the remark \ref{rem:move} to simplify the right hand sides.

From all the previous equations a more efficient algorithm can be written. The efficiency increment comes from the fact that several (gathered) states can be updated together. 

The following algorithm takes into account only equations \eqref{eq:sr}, \eqref{eq:sr5}, \eqref{eq:sr1}, \eqref{eq:sr55} and \eqref{eq:sr51}, but arrives to the solution of the game in the sense that, in all the other cases the optimal action is either $\parar$ or $\roll$, so the value coincides with $V_0(\tau,n)$. This fact can be verified directly in Table \ref{table:value}. We use the notations $V_1(\tau,n)=V(\tau,[0,2,0],n)$, $V_{55}(\tau,n)=V(\tau,[2,1,0],n)$ and $V_{51}(\tau,n)=V(\tau,[1,2,0],n)$. Equations can be rewritten in a more convenient way as follows:
\begin{align}
V_5(\tau,n)&=\tau\vee \vr(\tau,n)\vee \vr(\tau-1,n+1),\quad\text{for $n=1,2,3$}
.\tag{$E_5'$} \\ 
V_1(\tau,n)&=\tau\vee \vr(\tau,n)\vee \vr(\tau-2,n+1),\quad\text{for $n=1,2,3$}
.\tag{$E_1'$} \\
V_{55}(\tau,n)&=\tau \vee \vr(\tau,n)\vee \vr(\tau-1,n+1)\vee \vr(\tau-2,n+2),\quad\text{for $n=1,2$}.\tag{$E_{55}'$}\\
V_{51}(\tau,n)&=\tau \vee \vr(\tau,n)\vee \vr(\tau-1,n+1)\vee \vr(\tau-3,n+2),\quad\text{for $n=1,2$}.\tag{$E_{51}'$}
\end{align}

Notes on the algorithm
\begin{itemize}
\item Consider the set of labels $\mathcal{L}=\{0,5,1,55,51\}$
\item  $\vr(\tau,n)$ should be computed following \eqref{eq:vroll} considering $V(\tau,i,n)=\tau$ for $\tau\geq 56$, $V(\tau,i,n)=V_0(\tau,n)$ for every $i$ not included in the set $$\{[2,0,0],[1,1,0],[0,2,0],[2,1,0],[1,2,0]\}$$ in which the introduced notation has to be taken into account. 
\item After the execution of the algorithm $V$ will have the optimal value and $A$ will have the optimal action.
\item For notation convenience we use $\vect_{\ell}(\tau,n)$ to denote $(V_{\ell}(\tau,n),A_{\ell}(\tau,n))$
\end{itemize}
\texttt{
\begin{enumerate}
\item for $\tau=55$ {downto} $0$ 
\item \qquad $\vect_0(\tau,5) \gets (\vr(\tau,5),\roll)$
\item \qquad for ${n}=4$ {downto} $1$ 
\item \qquad\qquad $aux \gets \vr(\tau,{n})$
\item \qquad\qquad for $\ell\in\mathcal{L} $ 
\item \qquad\qquad\qquad if $aux\leq \tau$ then 
\item \qquad\qquad\qquad\qquad ${\vect_\ell(\tau,{n})\gets (\tau,\parar)}$ 
\item \qquad\qquad\qquad else
\item \qquad\qquad\qquad\qquad ${\vect_\ell(\tau,{n})\gets (aux,\roll)}$ 
\item \qquad\qquad\qquad\qquad if $n\geq 2$ and $V_5(\tau+1,{n}-1)<aux$ then 
\item \qquad\qquad\qquad\qquad\qquad $\vect_5(\tau+1,{n}-1)\gets (aux,\resign{5})$\label{5}
\item \qquad\qquad\qquad\qquad if $n\geq 2$ and  $V_1(\tau+2,{n}-1)<aux$ then 
\item \qquad\qquad\qquad\qquad\qquad $\vect_1(\tau+2,{n}-1) \gets (aux,\resign{1})$\label{1}
\item \qquad\qquad\qquad\qquad if $3\geq n\geq 2$ and $V_{55}(\tau+1,{n}-1)<aux$ then
\item \qquad\qquad\qquad\qquad\qquad $\vect_{55}(\tau+1,{n}-1)\gets (aux,\resign{55})$\label{55}  
\item \qquad\qquad\qquad\qquad if $n\geq 3$ and $V_{55}(\tau+2,{n}-2)<aux$ then
\item \qquad\qquad\qquad\qquad\qquad $\vect_{55}(\tau+2,{n}-2)\gets (aux,\resign{55})$\label{55}  
\item \qquad\qquad\qquad\qquad if $3\geq n\geq 2$ and $V_{51}(\tau+1,{n}-1)<aux$ then
\item \qquad\qquad\qquad\qquad\qquad $\vect_{51}(\tau+1,{n}-1)\gets (aux,\resign{51})$\label{51}  
\item \qquad\qquad\qquad\qquad if $n\geq 3$ and $V_{51}(\tau+3,{n}-2)<aux$ then
\item \qquad\qquad\qquad\qquad\qquad $\vect_{51}(\tau+3,{n}-2)\gets (aux,\resign{51})$\label{51}  
\end{enumerate}
}

\section*{Acknowledgement} First author is part of the SRI-UQ Center for Uncertainty Quantification, KAUST.

\newpage
\pagestyle{empty}
\begin{table}
\begin{center}
\begin{tabular}{c|c|c|r|rrrrr}
\hline
{\footnotesize Sc. dice} & Conf. 		& {\footnotesize Comb.} &  {\footnotesize Score} 		&\multicolumn{5}{|c}{Frequencies} \\
$s(i)$ & $f,o,t$  	& $c(i)$	&	  $s(i)$  			& {\footnotesize $f(1,i)$}   & {\footnotesize $f(2,i)$} & {\footnotesize $f(3,i)$}  & {\footnotesize $f(4,i)$}  & {\footnotesize $f(5,i)$}\\
\hline
0 & $0,0,0$ &   ${-}$ & $0$  & $4$  & $16$  & $60$  & $204$  & $600$ \\
\hline
\multirow{2}{*}{1}	 &  $1,0,0$ &1&   $1$  & $1$  & $8$  & $48$  & $240$  & $1020$ \\
&  $0,1,0$ &1&    $2$  & $1$  & $8$  & $48$  & $240$  & $1020$ \\
 \hline
\multirow{3}{*}{2}	 & $2,0,0$ &2&     $2$  & {\relax}  & $1$  & $12$  & $96$  & $600$ \\
& $1,1,0$ &2&     $3$  & {\relax}  & $2$  & $24$  & $192$  & $1200$ \\
& $0,2,0$ &2&     $4$  & {\relax}  & $1$  & $12$  & $96$  & $600$ \\
\hline
\multirow{8}{*}{3} & $3,0,0$ &3&    $10$  & {\relax}  & {\relax}  & $1$  & $16$  & $160$ \\
& $2,1,0$ &3&     $4$  & {\relax}  & {\relax}  & $3$  & $48$  & $480$ \\
& $1,2,0$ &3&     $5$  & {\relax}  & {\relax}  & $3$  & $48$  & $480$ \\
& $0,3,0$ &3&     $20$  & {\relax}  & {\relax}  & $1$  & $16$  & $160$ \\
& $0,0,2$ &1&    $4$  & {\relax}  & {\relax}  & $1$  & $13$  & $106$ \\
& $0,0,3$ &1&    $6$  & {\relax}  & {\relax}  & $1$  & $13$  & $106$ \\
& $0,0,4$ &1&    $8$  & {\relax}  & {\relax}  & $1$  & $13$  & $106$ \\
& $0,0,6$ &1&   $12$  & {\relax}  & {\relax}  & $1$  & $13$  & $106$ \\
 \hline
\multirow{13}{*}{4} &  $4,0,0$ &4&    $11$  & {\relax}  & {\relax}  & {\relax}  & $1$  & $20$ \\
&  $3,1,0$ &4&    $12$  & {\relax}  & {\relax}  & {\relax}  & $4$  & $80$ \\
&  $2,2,0$ &4&    $6$  & {\relax}  & {\relax}  & {\relax}  & $6$  & $120$ \\
&  $1,3,0$ &4&    $21$  & {\relax}  & {\relax}  & {\relax}  & $4$  & $80$ \\
&  $0,4,0$ &4&    $22$  & {\relax}  & {\relax}  & {\relax}  & $1$  & $20$ \\
&  $1,0,2$ &2&    $5$  & {\relax}  & {\relax}  & {\relax}  & $4$  & $65$ \\
&  $1,0,3$ &2&    $7$  & {\relax}  & {\relax}  & {\relax}  & $4$  & $65$ \\
&  $1,0,4$ &2&    $9$  & {\relax}  & {\relax}  & {\relax}  & $4$  & $65$ \\
&  $1,0,6$ &2&    $13$  & {\relax}  & {\relax}  & {\relax}  & $4$  & $65$ \\
&  $0,1,2$ &2&    $6$  & {\relax}  & {\relax}  & {\relax}  & $4$  & $65$ \\
&  $0,1,3$ &2&    $8$  & {\relax}  & {\relax}  & {\relax}  & $4$  & $65$ \\
&  $0,1,4$ &2&    $10$  & {\relax}  & {\relax}  & {\relax}  & $4$  & $65$ \\
&  $0,1,6$ &2&   $14$  & {\relax}  & {\relax}  & {\relax}  & $4$  & $65$ \\
 \hline
\multirow{18}{*}{5} &  $5,0,0$ &5&    $12$  & {\relax}  & {\relax}  & {\relax}  & {\relax}  & $1$ \\
&  $4,1,0$ &5&    $13$  & {\relax}  & {\relax}  & {\relax}  & {\relax}  & $5$ \\
&  $3,2,0$ &5&     $14$  & {\relax}  & {\relax}  & {\relax}  & {\relax}  & $10$ \\
&  $2,3,0$ &5&     $22$  & {\relax}  & {\relax}  & {\relax}  & {\relax}  & $10$ \\
&  $1,4,0$ &5&     $23$  & {\relax}  & {\relax}  & {\relax}  & {\relax}  & $5$ \\
&  $0,5,0$ &5&     $24$  & {\relax}  & {\relax}  & {\relax}  & {\relax}  & $1$ \\
&  $2,0,2$ &3&     $6$  & {\relax}  & {\relax}  & {\relax}  & {\relax}  & $10$ \\
&  $2,0,3$ &3&     $8$  & {\relax}  & {\relax}  & {\relax}  & {\relax}  & $10$ \\
&  $2,0,4$ &3&     $10$  & {\relax}  & {\relax}  & {\relax}  & {\relax}  & $10$ \\
&  $2,0,6$ &3&     $14$  & {\relax}  & {\relax}  & {\relax}  & {\relax}  & $10$ \\
&  $1,1,2$ &3&     $7$  & {\relax}  & {\relax}  & {\relax}  & {\relax}  & $20$ \\
&  $1,1,3$ &3&     $9$  & {\relax}  & {\relax}  & {\relax}  & {\relax}  & $20$ \\
&  $1,1,4$ &3&     $11$  & {\relax}  & {\relax}  & {\relax}  & {\relax}  & $20$ \\
&  $1,1,6$ &3&     $15$  & {\relax}  & {\relax}  & {\relax}  & {\relax}  & $20$ \\
&  $0,2,2$ &3&     $8$  & {\relax}  & {\relax}  & {\relax}  & {\relax}  & $10$ \\
&  $0,2,3$ &3&     $10$  & {\relax}  & {\relax}  & {\relax}  & {\relax}  & $10$ \\
&  $0,2,4$ &3&     $12$  & {\relax}  & {\relax}  & {\relax}  & {\relax}  & $10$ \\
&  $0,2,6$ &3&     $16$  & {\relax}  & {\relax}  & {\relax}  & {\relax}  & $10$ \\
 \hline
\end{tabular}
\end{center}
\label{table:freq}
\caption{Configurations obtained in one roll of the Ten Thousand game}
\end{table}

\begin{table}
\begin{threeparttable}
\label{table:value}
\begin{footnotesize}
\begin{tabular}{c|c|c|ccc|ccc||c|c}
\multirow{3}{*}{$\tau$}	&\multicolumn{1}{|c|}{$n=5$}	&$n=4$	&\multicolumn{3}{|c|}{$n=3$}	&	\multicolumn{3}{|c||}{$n=2$}	&	\multirow{3}{*}{$\tau$}&	{{\footnotesize $n=5$}}	\\
	& \multirow{2}{*}{all $i$}&\multirow{2}{*}{all $i$}	& {\footnotesize $[2,0,0]$} &\multirow{2}{*}{{\footnotesize $[0,2,0]$}} & \multirow{2}{*}{{\footnotesize other $i$}}	&	 \multirow{2}{*}{{\footnotesize $[2,1,0]$}} & \multirow{2}{*}{{\footnotesize $[1,2,0]$}} &  {\footnotesize $[2,0,0]$}	&		&	\multirow{2}{*}{all $i$	}\\
	
	& 	&	& {\footnotesize $[1,1,0]$}&  & & & &{\footnotesize $[1,1,0]$}	&		&	\\
\hline
0	&	{\footnotesize 5.8721}	&	--	&	--	&	--	&	--	&	--	&	--	&	--	&	28	&	30.181	\\
1	&	6.607	&	4.338	&	--	&	--	&	--	&	--	&	--	&	--	&	29	&	31.102	\\
2	&	7.355	&	5.021	&	4.338	&	--	&	3.447	&	--	&	--	&	--	&	30	&	32.022	\\
3	&	8.107	&	5.743	&	5.021	&	--	&	4.123	&	--	&	--	&	3.447	&	31	&	32.943	\\
4	&	8.883	&	6.476	&	5.743	&	5.021	&	4.837	&	5.021	&	--	&	4.123	&	32	&	33.864	\\
5	&	9.713	&	7.212	&	6.476	&	5.743	&	5.552	&	5.743	&	5.021	&		&	33	&	34.785	\\
6	&	10.553	&	8.001	&	7.212	&	6.476	&	6.266	&	6.476	&		&		&	34	&	35.706	\\
7	&	11.394	&	8.840	&	8.001	&	7.212	&		&	7.212	&		&		&	35	&	36.627	\\
8	&	12.235	&	9.679	&	8.840	&	8.001	&		&	8.001	&		&		&	36	&	37.548	\\
9	&	13.077	&	10.517	&	9.679	&		&		&		&		&		&	37	&	38.469	\\
10	&	13.918	&	11.357	&	10.517	&		&		&		&		&		&	38	&	39.390	\\
11	&	14.779	&	12.196	&	11.357	&		&		&		&		&		&	39	&	40.312	\\
12	&	15.655	&	13.035	&	12.196	&		&		&		&		&		&	40	&	41.233	\\
13	&	16.534	&	13.875	&	13.035	&		&		&		&		&		&	41	&	42.154	\\
14	&	17.413	&	14.715	&		&		&		&		&		&		&	42	&	43.075	\\
15	&	18.291	&	15.554	&		&		&		&		&		&		&	43	&	43.997	\\
16	&	19.170	&	16.394	&		&		&		&		&		&		&	44	&	44.919	\\
17	&	20.060	&	17.234	&		&		&		&		&		&		&	45	&	45.840	\\
18	&	20.972	&	18.073	&		&		&		&		&		&		&	46	&	46.762	\\
19	&	21.893	&		&		&		&		&		&		&		&	47	&	47.684	\\
20	&	22.814	&		&		&		&		&		&		&		&	48	&	48.607	\\
21	&	23.734	&		&		&		&		&		&		&		&	49	&	49.529	\\
22	&	24.655	&		&		&		&		&		&		&		&	50	&	50.452	\\
23	&	25.576	&		&		&		&		&		&		&		&	51	&	51.375	\\
24	&	26.497	&		&		&		&		&		&		&		&	52	&	52.298	\\
25	&	27.418	&		&		&		&		&		&		&		&	53	&	53.221	\\
26	&	28.339	&		&		&		&		&		&		&		&	54	&	54.144	\\
27	&	29.260	&		&		&		&		&		&		&		&	55	&	55.066	\\
\hline
$\action$ & $\roll$ & $\roll$ & $\resign{5}$ & $\resign{1}$ & $\roll$ & $\resign{55}$ & $\resign{51}$ & $\roll$ & $\action$  & $\roll$ \\
\end{tabular}
\end{footnotesize}
\caption{Value function and optimal strategy}
\begin{tablenotes}
\item[---] References:
\item[-] Cells with ``--'' are unreachable.
\item[-] The missing values mean that the optimal action is $\parar$, so the value coincides with $\tau$. 
\item[-] Column $n=1$ (and also $n=2$ for ``other $i$") is omitted, since the optimal action is always $\parar$.
\item[-] The optimal action for the present values on the table depend only in the column and is indicated in the last row.
\item[---] Observations:
\item[-] All present values are greater than $\tau$, otherwise the optimal action would be $\parar$
\item[-] Values in columns for which the optimal action is a move appear also in other columns (the destination of the move). These columns could be completed ``by hand''.
\end{tablenotes}
\end{threeparttable}
\end{table}

\end{document}